\documentclass[10pt]{amsart}
\usepackage{amssymb,latexsym,amsmath,amsthm,enumitem,mathrsfs,geometry}
\usepackage{fullpage}
\usepackage{fancyhdr}
\usepackage{comment}
\usepackage{mathrsfs}
\usepackage{hyperref}
\usepackage{lineno}
\usepackage{diagbox}
\usepackage{array}
\usepackage{tikz}
\usetikzlibrary{arrows}
\usetikzlibrary{positioning}
\usepackage{float}
\newtheorem{thm}{Theorem}[section]
\newtheorem*{thm*}{Theorem}

\newtheorem{lemma}[thm]{Lemma}

\newcommand{\beq}{\begin{equation}}
\newcommand{\eeq}{\end{equation}}

\newtheorem{remark}{Remark}
\usepackage[titletoc,title]{appendix}

\newcommand{\kommentar}[1]{}
\newtheorem{theorem}{Theorem}%[section]
 \newtheorem{corollary}{Corollary}%[section]

\newtheorem{Proposition}{Proposition}
\newtheorem*{remark*}{Remark}

\definecolor{pink}{rgb}{1,.2,.6}
\definecolor{orange}{rgb}{0.7,0.3,0}
\definecolor{blue}{rgb}{.2,.6,.75}
\definecolor{green}{rgb}{.4,.7,.4}
\definecolor{purple}{RGB}{127,0,255}

\begin{document}
\numberwithin{equation}{section}

\title{A lower bound for classical Kloosterman sums and an application}

\author[Stephan]{Stephan Baier}
\address{Ramakrishna Mission Vivekananda University, Belur Math, Howrah, West Bengal 711202}
\email{stephanbaier2017@gmail.com}
 \author[Das]{Jishu Das}
 \address{Department of Applied Sciences and Humanities, Mathematics Section, 
 COEP Technological University, Wellesley Road, Shivajinagar, Pune, Maharashtra 411005, India.}
\email{dasj.maths@coeptech.ac.in}
\author[Jewel]{Jewel Mahajan}
\address{School of Mathematical Sciences, National Institute of Science Education and Research, Bhubaneswar, an OCC of Homi Bhabha National Institute, P. O. Jatni, Khurda 752050, Odisha, India}
\email{jewelmahajan421@gmail.com}


\date{\today}

\begin{abstract} 
We present a lower bound for the classical Kloosterman sum $S(a,b;c)$ where $(ab,c)=1$ and $c$ is an odd integer. We apply this lower bound for Kloosterman sums to derive an explicit lower bound in Petersson's trace formula, subject to a given condition. Consequently, we achieve a modified version of \cite[Theorem 1.7]{JS}, where weight $k$ and level $N$ are permitted to vary independently.  Using this modified version, we get a lower bound for a weighted trace of the Hecke operator $T_n$ acting on the space $S_k(N)$, of cusp forms of weight $k$ and level $N$ with $(n,N)=1$.
\end{abstract}
\maketitle 

\section{Introduction}
The classical Kloosterman sum $S(a,b;c)$ is given by $$S(a,b;c)=\sum_{\substack{x (\text{mod} \hspace{1 mm} c)\\ \gcd(x,c)=1}} e^{2\pi i\left(\frac{ax+b\overline{x}(c)}{c}\right)},$$   
where   $a,b$ are integers, $c$ is a natural number and $\overline{x}(c)$ denotes the multiplicative inverse of $x$   modulo $c$. 
The sums for the special case of  $a=0 $ or $b=0 $  are called Ramanujan sums. Kloosterman sums appear in various problems in number theory. One of the important places of its occurrence is in the Kuznetsov trace formula \cite[Theorem 1]{Kuz} and Petersson's trace formula.  
 Recently, the non-vanishing aspect of Kloosterman sums has been used in \cite[Theorem 2]{BDS} and \cite{JD} to obtain discrepancy results for the distribution of eigenvalues of Hecke operators. Whenever we have non-vanishing of a sum, a natural question to ask is whether we can obtain a lower bound for the same. 
To the best of our knowledge, no such lower bound for Kloosterman sums has been established in the literature. In the next section, we answer this question for $S(a,b;c),$ where $(ab,c)=1$.   
 First, we obtain the following lower bound for classical Kloosterman sums. 
 \begin{theorem}\label{Main1}
       Let $c$ be an odd natural number such that $c=du$, where $u$ is squarefree and $d$ is a powerful number with $(d,u)=1$. Let $a,b$ be integers such that $(ab,c)=1$ and $ab$ is a quadratic residue modulo each prime dividing $d$ (and hence, $ab$ is a square modulo $d$ by Hensel's lemma and Chinese remainder theorem). 
       Then 
       $$
       \left| S(a,b;c) \right|\geq \frac{2^{\omega(d)}}{\sqrt{d}}\left( \frac{1}{\left(\tau(u)\sqrt{u}\right)^{\phi(u)/\tau(u)-1}}\right),
       $$  
       where $\omega(d)$ is the number of distinct prime divisors of $d$,  $\tau(u)$ is the number of positive divisors of $u,$ and $\phi(u)$ denotes the Euler's phi function of $u$. 
       If $ab$ is a quadratic non-residue modulo a prime dividing $d$ (and hence $ab$ is not a square modulo $d$), then $S(a,b;c)=0$. 
   \end{theorem}
   
   Let $S_k(N)$ denote the space of cusp forms of even weight $k\geq 2$ and level $N.$ The discrepancy between two probability measures involving  eigenvalues of cusp forms
on a given closed interval has been a well-studied topic in number theory.
Upper bounds (see \cite[Theorem 2]{MSeffective}, \cite[Theorem 1]{TW}, \cite[Theorem 2]{SZ}) and lower bound estimates (see \cite[Theorem 1.1, Theorem 1.6]{JS}, \cite[Theorem 1, Theorem 2]{JD}) of the
discrepancy of two probability measures are of particular interest due to its various
applications to other problems. For upper bounds, the key technique is the Eichler-Selberg trace formula (\cite{Sel}) and for lower bounds, one of the key ingredients is  Petersson's trace formula (\cite{PH}).
 The Bessel functions of the first kind $J_a(x)$ for $a\geq 0$, is defined by the following  power series representation 
$$
J_a(x)=\sum_{j=0}^{\infty } \frac{(-1)^j}{\Gamma(j+1)\Gamma(a+j+1)}\left(\frac{x}{2}\right)^{a+2j},
$$
where $\Gamma(x) $ denotes the Gamma function evaluated at $x$. 
We obtain the following version of Petersson's trace formula for the space $S_k(N)$ as an application to Theorem~\ref{Main1}.  
\begin{theorem}\label{main theorem sec 1}
  Let $m,n,N$ be natural numbers with $(mn,N)=1$ and $N$ odd. Let $H_0=1-\frac{4}{9}-\log\left(\frac{9}{5}\right)\approx -0.03223$, $D_0 \in \left(e^{H_0},1\right)$ be a real number ($e^{H_0}\approx 0.96828)$. Let  $N=du$ where $d$ is powerful and $u$ is squarefree, where $(d,u)=1$. Let $k_0(N)= 2+ \left\lfloor\frac{\log 7-\log H(N)}{\log A_0}\right\rfloor,$ where $A_0=D_0e^{-H_0},$ $H(N)=2^{\omega(N)} \frac{G(N)}{N^{3/2}}$ with $G(N)= \frac{u}{(\tau(u)^{2} u)^{\phi(u)/2\tau(u)}}.$  Let  $k\geq \max(22, k_0(N)) $,
 $\frac{4\pi\sqrt{mn}}{N}\in \left[ D_0(k-1), (k-1) \right] $ and assume that $mn$ is a quadratic residue modulo each prime dividing $d$. 
Let  $\Delta_{k, N}(m,n)$ be  as defined in equation \eqref{DeltakN}, then 
    $$\left| \Delta_{k,N}(m,n)-\delta(m,n)  \right|\geq   7.99 \pi J_{k-1}(k-1) e^{(k-1)\left(1-\frac{4}{9}-\log(\frac{9}{5})\right)}. 
    $$ 
\end{theorem}
Some key features of Theorem~\ref{main theorem sec 1} are as follows:
\begin{itemize}
\item Theorem~\ref{Main1}  gives us an explicit lower bound for a Kloosterman sum appearing in $\Delta_{k, N}(m,n)$. We then use explicit bounds for the $J$-Bessel function as given in Lemma~\ref{Bessel1} and Lemma~\ref{Bessel2}. We emphasise that  Lemma~\ref{Bessel2}, which follows from a direct application of \cite[Theorem 1]{KI}, is crucial for deriving an explicit lower bound for $|\Delta_{k,N}(m,n)-\delta(m,n)|$. One of the main purposes of \cite{KI} is to derive an explicit expression for $J_k(x)$ for a transition range of $x$  and for any $k\geq 2.$ 
    \item The theorems \cite[Theorem 1.7]{JS} and \cite[Theorem 4]{JD}, assume $N$ to be fixed and provide an asymptotic for $|\Delta_{k,N}(m,n)-\delta(m,n)|,$ as $k\rightarrow \infty. $  However, 
    Theorem~\ref{main theorem sec 1} is applicable for every $k$ and $N$, provided they satisfy certain hypotheses. For a fixed level $N,$ \cite[Corollary 2.3]{ILS} gives an asymptotic for $\Delta_{k,N}(m,n)$ with the condition $\frac{4\pi\sqrt{mn}}{N}\leq \frac{k}{3}, $ whereas Theorem~\ref{main theorem sec 1} gives us an asymptotic for the condition $D_0(k-1)\leq \frac{4\pi\sqrt{mn}}{N}\leq (k-1).$  
    \item The size of the interval in which $\frac{4\pi\sqrt{mn}}{N}$ lies is $(1-D_0)k$ as compared to $2k^{\frac{1}{3}}$ of \cite[Theorem 1.7]{JS}, which weakens the lower bound to  $e^{(k-1)\left(1-\frac{4}{9}-\log(\frac{9}{5})\right)}$ as compared to $k^{-\frac{1}{3}}$. However, for a shorter range given by $\left[(k-1),(k-1)+(k-1)^\frac{1}{3} \right],$  we obtain back the lower bound of order  $(k-1)^{-\frac{1}{3}},$ which is given by Theorem~\ref{main theorem sec 2}. 
\end{itemize}

The following corollary is immediate. 
\begin{corollary}\label{cor 1}
    Let $0< \epsilon <0.4473.$  Let $m,n,N$ be natural numbers with $(mn,N)=1$ and $N$ odd. Let $H_0,D_0,k_0(N)$ be as defined in Theorem~\ref{main theorem sec 1}.
    % 
    % satisfying $\left( D_0 e^{-H_0} \right)^{k-1}\geq \frac{7}{H(N)}$. 
     Let $N=du$ where $d$ is powerful and $u$ is squarefree with $(d,u)=1$.  If  $\frac{4\pi\sqrt{mn}}{N}\in \left[D_0(k-1), (k-1) \right] $ and $mn$ is a quadratic residue modulo each prime dividing $d$, 
     then  for all $k\geq \max\left( k_0(N), 1+\left( \frac{3}{\epsilon}\right)^{3/2}\right), $ we have 
  $$\left| \Delta_{k,N}(m,n)-\delta(m,n)  \right|\geq    
   7.99 \pi\frac{(0.4473-\epsilon)}{(k-1)^{1/3}} e^{(k-1)\left(1-\frac{4}{9}-\log(\frac{9}{5})\right)} .
    $$
\end{corollary}

\begin{theorem}\label{main theorem sec 2}
 Let $m,n,N$ be natural numbers with $(mn,N)=1$ and $N$ odd.   Let $N=du$ where $d$ is powerful and $u$ is squarefree, where $(d,u)=1$. Let $\frac{4\pi\sqrt{mn}}{N}\in \left[ (k-1), (k-1)+(k-1)^{\frac{1}{3}} \right] $ and assume that $mn$ is a quadratic residue modulo each prime that divides the powerful part of $N$. Let $E_0$ be a positive constant less than $0.327H(N)$ (as defined in Theorem~\ref{main theorem sec 1}) and $k_{1}(N):=2+\left\lfloor \frac{\log (1.632)-\log (0.327H(N)-E_0)}{-H_0}\right\rfloor$. Then,  for all $k\geq \max(180, k_{1}(N)),$ we have 
    $$\left| \Delta_{k,N}(m,n)-\delta(m,n)  \right|\geq    \frac{2\pi E_0}{(k-1)^{\frac{1}{3}}}.$$    
\end{theorem}
\begin{remark}\label{remark k1/3}
Let $C_0$ be a positive constant less than $0.327,$ and $k_{1}(N):=2+\left\lfloor \frac{\log (1.632 /C_0)-\log H(N)}{-H_0}\right\rfloor$. With the assumptions as in Theorem~\ref{main theorem sec 2}, we have 
    $$\left| \Delta_{k,N}(m,n)-\delta(m,n)  \right|\geq    \frac{2\pi (0.327-C_0) H(N)}{(k-1)^{\frac{1}{3}}},$$
    for all $k\geq \max(180, k_{1}(N)).$
\end{remark}

\subsection*{Organisation of the article} The article is organised as follows. In Section \ref{sec2}, we discuss proofs for  the lower bound of classical Kloosterman sums and for Corollary ~\ref{Kloosterman general}. 
In Sections \ref{sec3}, we prove Theorems \ref{main theorem sec 1}, \ref{main theorem sec 2} and Corollary ~\ref{D1n}. 
In the very end of Section \ref{sec3}, we provide a table for the values of $H(N),k_0(N) $ and $k_1(N)$ for some specific values of $N.$
\section{Lower bound for Kloosterman sums}\label{sec2}
In this section, we primarily prove Theorem~\ref{Main1} by considering squarefree and odd prime power cases separately, and we consider the general case in Section \ref{gs}. 
\subsection{Squarefree case}
\begin{Proposition}\label{prime}
Let $c $ be a squarefree positive integer such that $(a,c)=1.$ Then 
$$|S(a,b;c)|\geq \frac{1}{\left(\tau(c)\sqrt{c}\right)^{\phi(c)/\tau(c)-1}},$$ 
\end{Proposition}
\begin{proof}
Let $\alpha=S(a,b;c)\in \mathbb{Q}(\zeta_{c})$, where $\zeta_c=e^{2\pi i/c}$.
An element of Gal$(\mathbb{Q}(\zeta_{c})/\mathbb{Q})$
sends $\zeta_c$ to $\zeta_c^{n}$ where $(n,c)=1.$ Hence, the conjugates of $\alpha$ are the Kloosterman sums $S(an,bn;c),$ where $(n,c)=1.$ In fact, we have $S(an,bn;c)=S(an^{2},b;c)=S(a,bn^{2};c).$
 For $c$ prime,  $S(a,b;c)$  is a linear combination of at most $c-1$ $c$-th roots of unity, but they are necessarily linearly independent over $\mathbb{Q}$. So if the linear combination becomes $0$, all the coefficients need to be $0$ which is not the case. 
Hence $S(a,b;c)\neq 0$ for $c$ prime. Using the multiplicativity of Kloosterman sums  (see Lemma~\ref{multiplicativity}), $S(a,b;c)\neq 0$ for squarefree $c$. 
Let $\mathcal{N}(\alpha)$ denote the norm of $\alpha\in \mathbb{Q}(\zeta_{c})$ over $\mathbb{Q}.$
Since $\alpha=S(a,b;c)$ is a non-zero algebraic integer in $\mathbb{Q}(\zeta_{c})$, its norm $\mathcal{N}(\alpha)$ is a non-zero integer and therefore,
$$1\leq |\mathcal{N}(\alpha)|=|\mathcal{N}(S(a,b;c))|=\prod_{\substack{(n,c)=1\\1 \leq  n \leq c}} |S(an^{2},b;c)|.$$
Hence, 
$$\displaystyle\prod_{\substack{(n,c)=1\\1\leq  n \leq c}} |S(an^{2},b;c)| \geq 1.$$
Let $\rho(c)$ be the number of solutions modulo $c$ of $x^2\equiv 1$ (mod $c$). We have $\rho(c)=\tau(c) $ for squarefree $c$. 
 Using Weil's bound together with $(a,c)=1$, we have
$$\frac{1}{|S(a,b;c)|^{\tau(c)}} \leq \displaystyle\prod_{\substack{(n,c)=1\\1< n \leq c\\n^{2}\not\equiv 1 (c) }} |S(an^{2},b;c)| \leq \displaystyle\prod_{\substack{(n,c)=1\\1< n \leq c\\n^{2}\not\equiv 1 (c)}} \left( \tau(c)\sqrt{c}\right)=\left(\tau(c)\sqrt{c}\right)^{\phi(c)-\tau(c)}.$$ 
Hence we obtain 
\begin{equation*}
|S(a,b;c)|\geq \frac{1}{\left(\tau(c)\sqrt{c}\right)^{\phi(c)/\tau(c)-1}}.
\end{equation*}
\end{proof}

\subsection{Odd prime power case}
Let $q$ be odd, and $\epsilon_q$ be $1$ or $i$ depending on whether $q \equiv$ $1$ or $-1$ modulo $4.$ For odd prime power moduli $q$, we have the following explicit evaluation of $S(a,b;q)$.

 \begin{Proposition}\cite[Chapter 12, Exercise 1]{IK}\label{prop 2}
   Let $q=p^\beta$ for some prime $p$ and $\beta\geq 2. $  Suppose $p$ does not divide $2ab$. If $ab$ is a quadratic residue modulo $p$, then 
   \begin{equation}
      S(a,b;q)= 2 \left( \frac{l}{q}\right) \sqrt{q} \text{ Re}  \left( \epsilon_q e\left(\frac{2l}{q}\right) \right),
   \end{equation}
   where $l^2\equiv ab$ (mod $q)$. If $ab$ is a quadratic non-residue modulo $p$, then we have $S(a,b;q)=0.$
 \end{Proposition}

 From the above Proposition, we deduce the following lower bound for $|S(a,b;q)|,$ if $q$ is an odd prime power.
 
\begin{Proposition}\label{primepower}
   Let $q=p^\beta$ for some prime $p$ and $\beta\geq 2. $
   Suppose $p$ does not divide $2ab$. If $ab$ is a quadratic residue modulo $p$, then 
     $$ |S(a,b;q)|\geq \frac{2}{\sqrt{q}}.$$
     If $ab$ is a quadratic non-residue modulo $p$, then $S(a,b;q)=0$.
\end{Proposition}
\begin{proof}
Assume that $ab$ is a quadratic residue modulo $p.$ Using Proposition \ref{prop 2}, we have 
\begin{equation*}
      S(a,b;q)= 2 \left( \frac{l}{q}\right) \sqrt{q} \text{ Re}  \left( \epsilon_q e\left(\frac{2l}{q}\right) \right),
   \end{equation*}
where \begin{equation} \label{conrel}
l^2\equiv ab\ (\mbox{mod } q).
\end{equation}
We start by noting that $q$ does not divide $l$ and hence $\left| \left( \dfrac{l}{q} \right)\right|=1.$  Let $l=qt+r$ where $0\leq r\leq q-1.$
 Note that $p \nmid 2ab$ implies that $p$ is odd and so is $q.$ We consider the following two cases depending on whether $q\equiv$ $1$ or $-1$ modulo $4.$
 
\textbf{Case 1:  $q\equiv 1 (\text{mod }4)$}

In this case, we have $\epsilon_q=1$. For $x\in \mathbb{R}$, let $\| x\|$ denote the distance from $x$ to the nearest integer. Using $|\sin \pi x|\geq 2\| x\|$ for all real $x$, we have  
    \begin{align*}
      |S(a,b;q)|=&2\left|\left(\frac{l}{q} \right)\sqrt{q}  \,\text{Re} \, \left( e\left(\frac{2l}{q}  \right) \right)   \right|
       =2\sqrt{q} \left| \cos \left( \frac{4\pi r}{q} \right) \right|\\
       =& 2\sqrt{q} \left| \sin \left( \frac{\pi}{2} - \frac{4 \pi r}{q} \right) \right| 
       \geq  4\sqrt{q} \left\|\frac{1}{2}-\frac{4r}{q} \right\|.
    \end{align*}
 
 \textbf{Case 2:  $q\equiv -1 (\text{mod }4)$}

In this case, we have $\epsilon_q=i$ and hence, we have
\begin{align*}
       |S(a,b;q)|=&2\left|  \left(\frac{l}{q} \right)\sqrt{q}  \,\text{Re} \, \left( ie\left(\frac{2l}{q}  \right) \right)   \right|
       \\=&2\sqrt{q} \left| \sin  \left( \frac{4\pi r}{q} \right) \right|
       \geq 4\sqrt{q} \left\|\frac{4r}{q} \right\|.
    \end{align*}
   Combining both cases, we obtain
   \begin{equation}\label{geqmin}
   |S(a,b;q)|\geq 4\sqrt{q} \min\left(  \left\| \frac{q-8r}{2q}  \right\|,  \left\|\frac{4r}{q} \right\| \right).
   \end{equation}
We first show that $\left\|\frac{q-8r}{2q} \right\| \neq 0.$ Suppose not, then  $\frac{q-8r}{2q}=t$ for some $t\in\{0,-1,-2,-3\}$ as $0\leq r\leq q-1.$   
We have $p \,|\, r^2 -ab$ which implies that $p$ divides $\frac{(1-2t)^2q^2-64 ab }{64}.$ This is an impossibility since $p$ does not divide $2ab.$
A similar argument yields $\left\|\frac{4r}{q} \right\|\neq 0.$
Proceeding now from the equation \eqref{geqmin}, we conclude 
that
 \begin{equation*}
   |S(a,b;q)|\geq 4\sqrt{q} \min\left(  \left\| \frac{q-8r}{2q}  \right\|,  \left\|\frac{4r}{q} \right\| \right)\geq 4\sqrt{q} \min\left(\frac{1}{2q},\frac{1}{q}\right)\geq \frac{2}{\sqrt{q}}.
   \end{equation*}
   \end{proof}
\subsection{General Case}\label{gs}
In this section, we apply the multiplicativity property of Kloosterman sums. 
For $m,n\in \mathbb{N} , $ $(m,n)=1, $ let $ \overline{m}(n)$ denote the multiplicative inverse of $m$ in $(\mathbb{Z}/ n\mathbb{Z})^*,$ that is,  $m \overline{m}(n)\equiv 1$ (mod $n$). 
We have the following factorization of Kloosterman sums into Kloosterman sums to prime power moduli.

\begin{Proposition}\cite[Lemma 3.1]{JD}\label{multiplicativity}
Let $c=q_1^{\beta_1} \cdots  q_t^{\beta_t},$
where $q_1,  \,\dots\, , q_t$ are distinct prime numbers,  $q_i^{\beta_i}=l_i,$ and   $$c_i=\frac{c}{\prod_{j=1}^i l_j}$$  for $1\leq i\leq t.$ 
Further let $m_0=1,$ $m_i=\prod_{j=1}^i\overline{l_j}({c_{j}})$ for $1\leq i\leq t-1.$  Then 
%\SBcom{What is the use of $n$?}
%\JDcom{It was not required. I have deleted it.}
\begin{equation*}\label{prodKlo}
S(a,b;c)=\prod_{i=1}^t S(a m_{i-1}\overline{c_i}({l_i}), bm_{i-1}\overline{c_i}({l_i});l_i).
\end{equation*}
\end{Proposition}
\begin{proof}
    The proof uses equation (1.59) of \cite{IK} repetitively which is given by 
\begin{equation*}
S(a,b;ll')=S(a\overline{l'}(l),b\overline{l'}(l);l)S(a\overline{l}(l'),b\overline{l}(l');l')
\end{equation*}
where $(l,l')=1$.
\end{proof}

\begin{remark} \label{quadres}
Assume the hypotheses in the above Proposition \ref{multiplicativity} are satisfied. We note that $ab$ is a quadratic residue modulo $q_i$ if and only if $a'b'$ with $a'=a m_{i-1}\overline{c_i}({l_i})$ and $b'=bm_{i-1}\overline{c_i}({l_i})$ is a quadratic residue modulo $q_i$. 
\end{remark}

Now we give a proof of Theorem~\ref{Main1}.

   \begin{proof}[\textbf{Proof of Theorem~\ref{Main1}}]
   We first factorise $S(a,b;c)$ using Proposition \ref{multiplicativity} into a Kloosterman sum with modulus $u$ and a product of Kloosterman sums with prime power moduli $l_i$ dividing $d$. Then we apply Proposition \ref{prime} to the Kloosterman sum with modulus $u$ and Proposition \ref{primepower} to the Kloosterman sums with moduli $l_i$, taking the condition $(ab,c)=1$ and Remark \ref{quadres} into account.   
   \end{proof}

   \begin{corollary}\label{Kloosterman general}
      Let $a,b \in \mathbb{Z}$ and $c$ be an odd natural number such that $c=du,$ where $u$ is squarefree and $d$ is a powerful number with $(d,u)=1,$ and $(ab,c)=1$. Assume that $ab$ is a quadratic residue modulo each prime dividing $d$.
       Then 
       $$
       \left| S(a,b;c) \right|\geq  \left(\frac{u^2}{c}\right)^{1/2} \frac{2^{\omega(c)}}{(\tau(u)^{2} u)^{\phi(u)/2\tau(u)}}.
       $$  

   \end{corollary}

       \begin{proof}
    The proof follows from the following observation:
    \begin{equation*}
        \begin{split}
            &  \frac{2^{\omega(d)}}{\sqrt{d}} \left( \frac{1}{\left(\tau(u)\sqrt{u}\right)^{\phi(u)/\tau(u)-1}} \right)\\
            \geq & \frac{2^{\omega(d)+\omega(u)}}{\sqrt{du}}  \left( \frac{1}{\left(\tau(u)\right)^{\phi(u)/\tau(u)}} \right)\left( \frac{1}{\left(\sqrt{u}\right)^{\phi(u)/\tau(u)-2}} \right)\\
            \geq &  \frac{2^{\omega(c)}}{\sqrt{c}}
             \frac{1}{\tau(u)^{\phi(u)/\tau(u)}} \frac{1}{u^{\phi(u)/2\tau(u)-1}},
        \end{split}
    \end{equation*}
    where we have used $\tau(u)=2^{\omega(u)}$ for squarefree $u.$
 
 \end{proof}
\section{An application}\label{sec3}
Let $k$ and $N$ be positive integers with $k$ even. Recall that 
$S_k(N)$ denotes the vector space of cusp forms of weight $k$ and level $N$.   For $(n,N)=1,$ the $n$-th (normalised) Hecke operator acting on $S_k(N)$ is given by $$T_n(f)(z)=n^{\frac{k-1}{2}}\sum_{ad=n,d>0}\frac{1}{d^k}\sum_{b \,(\text{mod} \, d) }f\Big(\frac{az+b}{d}\Big).$$ 

Let $\mathcal{F}_k(N)$ be an orthonormal basis of $S_k(N)$ consisting only of joint eigenfunctions of the Hecke operators $T_n$. Since $S_k(N)$ is finite dimensional, $|\mathcal{F}_k(N)|$ is finite. For $f\in S_k(N),$  the Fourier expansion of $f$ at the cusp $\infty$ is given by $$f(z)=\sum_{n=1}^{\infty} a_n(f)n^{\frac{k-1}{2}}  e^{2\pi i nz}.$$ We denote $\lambda_n(f)$ to be the $n$-th normalised Hecke eigenvalue of $f$, that is, $T_n(f)=\lambda_n(f) f$. The Fourier coefficient $a_n(f)$ and $\lambda_n(f) $ are related by the condition  $a_n(f)=a_1(f)\lambda_n(f).$

Petersson (see 
 \cite{PH}) expressed a weighted sum of $\overline{a_m(f)}a_n(f)$  over $f$ with $f\in \mathcal{F}_k(N)$ in terms of  Bessel function of the first kind and  Kloosterman sums. Let  us consider the Fourier coefficient $a_n(f)$ with harmonic weights $$\rho_n(f)=\Bigg(\frac{\Gamma(k-1)}{ (4\pi )^{k-1}}\Bigg)^{\frac{1}{2}}   a_n(f).$$ The sum under consideration is
\begin{equation}\label{DeltakN}
\Delta_{k,N}(m,n)=\sum_{f\in \mathcal{F}_k(N)}\overline{\rho_m(f)}\rho_n(f).
\end{equation}
Putting $m=1 $ in the above equation, we have 
\begin{align*}
    \Delta_{k,N}(1,n)=&\sum_{f\in \mathcal{F}_k(N)}\overline{\rho_1(f)}\rho_n(f) \\
   =&\sum_{f\in \mathcal{F}_k(N)} \left(\frac{\Gamma(k-1)}{ (4\pi )^{k-1}}\right) \overline{a_1(f)}a_n(f)=\sum_{f\in \mathcal{F}_k(N)} \left(\frac{\Gamma(k-1)}{ (4\pi )^{k-1}}\right) |a_1(f)|^2 \lambda_n(f), 
\end{align*}
since  $a_n(f)=a_1(f)\lambda_n(f)$.
Hence, 
\begin{equation}\label{justdelta(1,n)}
\Delta_{k,N}(1,n)=\sum_{f\in \mathcal{F}_k(N)} \left(\frac{\Gamma(k-1)}{ (4\pi )^{k-1}}\right) |a_1(f)|^2\lambda_n(f).
\end{equation}
The above equation demonstrates how Petersson's trace formula is indeed a weighted trace formula for Hecke operators.

  We use the following version of Petersson's trace formula given in \cite[Proposition 2.1]{ILS}.
  \begin{Proposition}[{{\cite[Proposition\ 2.1]{ILS}}}] \label{PTf}
 Let $k,m,n, N$ be natural numbers with $k$ even. We have $$\Delta_{k,N}(m,n)=\delta(m,n)+2\pi i^k \sum_{N|c,c>0} \frac{S(m,n;c)}{c}J_{k-1}\Big(\frac{4\pi\sqrt{mn}}{c}\Big),$$
 where $\delta(m,n)$ is the Kronecker diagonal symbol, $J_{k-1}(x)$ is the Bessel function of the first kind, and $S(m,n;c)$ is the classical Kloosterman sum.
\end{Proposition}
We will use the following two lemmas concerning the $J$-Bessel function. 
\begin{lemma}\label{Bessel1}
 If $\nu>0$ and $0<x\leq 1$, we have $$1\leq \frac{J_\nu(\nu x)}{x^\nu J_\nu (\nu)}\leq e^{\nu(1-x)}.$$ 
\end{lemma}
\begin{proof}
    We refer to \cite{Paris} for a proof. 
\end{proof}
\begin{lemma}\label{Bessel2}
    Let $\epsilon\in (0,C_0)$ be arbitrary and $A(\epsilon) := C_0-\epsilon$,    where $C_{0}=\frac{2^{1/3}}{3^{2/3}\Gamma(\frac{2}{3})}= 0.4473\dots .$
 Then, for all $\nu \geq \left( \frac{3}{\epsilon}\right)^{3/2},$ we have
    \begin{equation}\label{eq:Bessel bound}
         \frac{A(\epsilon)}{\nu^{1/3}}\leq J_{\nu}(\nu)\leq \frac{C_0}{\nu^{1/3}}+\frac{3}{\nu}.
    \end{equation}
      In particular, when $\epsilon =0.4470,$ 
      $$\frac{A(\epsilon)}{\nu^{1/3}}\leq J_{\nu}(\nu)\leq \frac{C_0}{\nu^{1/3}}+\frac{3}{\nu}$$
       for all $\nu\geq 18.$ 

\end{lemma}
\begin{proof}
    The proof follows from a direct application of \cite[Theorem~1]{KI} with $z=0,$ which implies that
  \begin{equation}\label{Illa}
    J_{\nu}(\nu) =\frac{2^{1/3}}{\nu^{1/3}} Ai(0)+ \frac{3\theta}{\nu},
    \end{equation}
    where $\theta$ is such that $|\theta|\leq 1,$ and $Ai(x)$ is the Airy function. We have
    $Ai(0) =\frac{1}{3^{2/3}\Gamma(\frac{2}{3})}.$ Therefore,
    $$-\frac{3}{\nu} \leq J_{\nu}(\nu) -\frac{C_0}{\nu^{1/3}}\leq \frac{3}{\nu}.$$
    For all $\nu \geq \left( \frac{3}{\epsilon}\right)^{3/2},$ we have  $\nu^{2/3} \geq \frac{3}{C_0-A(\epsilon)},$ that is, $\frac{C_0}{\nu^{1/3}} -\frac{3}{\nu}  \geq \frac{A(\epsilon)}{\nu^{1/3}}$, and this proves the lower bound for $J_{\nu}(\nu).$
    \end{proof} 

We give a detailed proof of the Theorem~\ref{main theorem sec 1} after an outline of the same.
\subsection*{Outline of the proof of Theorem~\ref{main theorem sec 1}}
We first break the infinite sum over $c$ in Proposition \ref{PTf}  into two parts given by $c=N$ and $c>N.$ For $c>N,$ applying the integral test, we give an upper bound using the exponential bound for the J-Bessel function. This is achieved in equation \eqref{|B|}. In the next step, we use Lemma~\ref{Bessel2} and Theorem~\ref{Main1} to get a lower bound for $\left|\Delta_{k,N}(m,n)-\delta(m,n)\right|.$

\begin{proof}[\textbf{Proof of Theorem~\ref{main theorem sec 1}}]
Applying Proposition \ref{PTf}, breaking the $c$-sum into the parts with $c=N$ and $c>N$, and writing $c=bN$, we have 
 \begin{align*}
      \Delta_{k,N}(m,n)&=\delta(m,n)+2\pi i^k\frac{S(m,n;N)}{N}J_{k-1}\Big(\frac{4\pi\sqrt{mn}}{N}\Big) +2\pi i ^k \sum_{b\geq 2} \frac{S(m,n;bN)}{bN}J_{k-1}\Big(\frac{4\pi\sqrt{mn}}{bN}\Big).
    \end{align*}
Using Lemma~\ref{Bessel1}, we get 
\begin{align*}
     \left| J_{k-1}\left(\frac{4\pi\sqrt{mn}}{bN}\right)\right|&=\left| J_{k-1}\left((k-1)\frac{4\pi\sqrt{mn}}{(k-1)bN}\right)\right|
      \leq e^{\nu(1-x)}x^{\nu}|J_\nu(\nu)|,
    \end{align*}
where $\nu=k-1$ and $x=\frac{4\pi\sqrt{mn}}{(k-1)bN}.$
For  $b\geq 2,$ and $k\geq 28,$ we have
\begin{equation}\label{8/9}
\frac{8}{9b}\leq\frac{4\pi\sqrt{mn}}{(k-1)bN}\leq \frac{10}{9b},
\end{equation}
which implies
\begin{align*}
    e^{\nu(1-x)}x^\nu&=e^{(k-1)(1-x+\log x)} \leq e^{(k-1)\big( 1-\frac{8}{9b}+\log \frac{10}{9b}\big)}.
\end{align*}
Therefore, using a trivial bound $|S(m,n;bN)|\leq \phi(bN), $ we obtain 
\begin{equation}\label{j}
\begin{split}
&\left|\sum_{b\geq 2}\frac{S(m,n;bN)}{bN}J_{k-1}\Big(\frac{4\pi\sqrt{mn}}{bN}\Big)\right|  \\
\leq & \sum_{b\geq 2}\left| J_{k-1}\Big(\frac{4\pi\sqrt{mn}}{bN}\Big)\right|
 \leq |J_{k-1}(k-1)| e^{(k-1)}\sum_{b\geq 2}e^{-(k-1)\big(\frac{8}{9b}+\log \frac{9b}{10}\big)}\\
 =& |J_{k-1}(k-1)| e^{(k-1)}\Bigg(\sum_{b=2}^4 e^{-(k-1)\big(\frac{8}{9b}+\log \frac{9b}{10}\big)} \Bigg)
 + |J_{k-1}(k-1)| e^{(k-1)}\Bigg(\sum_{b\geq 5}e^{-(k-1)\big(\frac{8}{9b}+\log \frac{9b}{10}\big)} \Bigg).
\end{split}
\end{equation}

Note that we have replaced \(\log\frac{10}{9b}\) by \(-\log\frac{9b}{10}\). Since \begin{equation}\label{breakon}
 e^{(k-1)}\Bigg(\sum_{b=2}^4 e^{-(k-1)\big(\frac{8}{9b}+\log \frac{9b}{10}\big)} \Bigg)\leq 3 e^{(k-1)\left(1-\frac{4}{9}-\log(\frac{9}{5})\right)},
\end{equation} 
 we  focus on  $\displaystyle\sum_{b\geq 5}e^{-(k-1)\big(\frac{8}{9b}+\log \frac{9b}{10}\big)} $ only.
 
 Let $F(x)=e^{-(k-1)\left(\frac{8}{9x}+\log \frac{9x}{10}\right)}$. Then $F$ is decreasing for $x>2.$  Using integral test, we obtain
 \begin{equation}\label{itest}
 \begin{split}
 \sum_{b\geq 5}e^{-(k-1)\big(\frac{8}{9b}+\log \frac{9b}{10}\big)} \leq & \int_4^\infty e^{-(k-1)\big(\frac{8}{9x}+\log \frac{9x}{10}\big)} \,dx\\
= &\frac{25}{18} \int_0^{\frac{2}{9} } e^{-(k-1)y}\left(\frac{5y}{4}\right)^{k-3}\,dy \\
\leq &\frac{25}{18} \int_0^{\frac{2}{9} }  \left(\frac{5}{18} \right)^{k-3} e^{-(k-1)y}  \,dy \\
=& \frac{25}{18}\left(\frac{5}{18} \right)^{k-3}\left( \frac{1-e^{\frac{-2(k-1)}{9}} }{(k-1)}\right)\leq \frac{25}{18}\left(\frac{5}{18} \right)^{k-3} \frac{1 }{(k-1)}. 
 \end{split}
 \end{equation}

Combining equations \eqref{j}, \eqref{breakon}, and \eqref{itest}, we have
\begin{equation}\label{finalineq}
    \begin{split}
\left|\sum_{b\geq 2}\frac{S(m,n;bN)}{bN}J_{k-1}\Big(\frac{4\pi\sqrt{mn}}{bN}\Big)\right| \leq 
|J_{k-1}(k-1)|\left(3e^{(k-1)\left(1-\frac{4}{9}-\log(\frac{9}{5})\right)}+\frac{25}{18}
 \frac{e^{2}}{(k-1)}\Big(\frac{5e}{18} \Big)^{k-3}\right).
    \end{split}
\end{equation}
On considering  equation \eqref{finalineq} and  equation \eqref{j}, we obtain
$$ \Delta_{k,N}(m,n)=\delta(m,n)+2\pi i^k \frac{S(m,n;N)}{N}J_{k-1}\left(\frac{4\pi\sqrt{mn}}{N} \right)+B,
$$
 where 
\begin{equation*}
    \begin{split}
        B=2\pi i ^k \sum_{b\geq 2} \frac{S(m,n;bN)}{bN}J_{k-1}\Big(\frac{4\pi\sqrt{mn}}{bN}\Big),
    \end{split}
\end{equation*} 
with
 \begin{equation}\label{|B|}
 \begin{split}
     |B|\leq &2\pi |J_{k-1}(k-1)|\left(3e^{H_0(k-1)}+
 \frac{25}{18}\frac{e^{2}}{(k-1)}\Big(\frac{5e}{18} \Big)^{k-3}\right)\\
 \leq & 6.01 \pi |J_{k-1}(k-1)| e^{H_0(k-1)},
 \end{split}
 \end{equation}
 where we use the fact that 
 $$ \frac{25}{18}\frac{e^{2}}{(k-1)}\Big(\frac{5e}{18} \Big)^{k-3}
\leq 0.005e^{H_0(k-1)},
$$
 for all $k \geq 22.$
    Since $\frac{4\pi\sqrt{mn}}{N(k-1)}\leq 1,$ using Lemma~\ref{Bessel1} together with the fact that $J_{k-1}(k-1)>0$ (this follows from Lemma~\ref{Bessel2}), we obtain 
    \begin{equation}\label{D0}
        J_{k-1}\left(\frac{4\pi\sqrt{mn}}{N}\right)\geq J_{k-1}\left( k-1\right)\left(\frac{4\pi\sqrt{mn}}{N(k-1)}\right)^{k-1}
      \geq D_0^{k-1}  J_{k-1}\left( k-1\right).
      \end{equation}
      Let $G(N)= \frac{u}{(\tau(u)^{2} u)^{\phi(u)/2\tau(u)}},$ for $N=du.$ Then, Corollary ~\ref{Kloosterman general} gives 
$$ \left| S(m,n;N) \right|\geq 2^{\omega(N)} \frac{G(N)}{N^{1/2}} .
       $$   
Combining the previous equation with Equation \eqref{D0}, we get 
       \begin{equation}\label{S(m,n;N)times J_k lower bound}
           \begin{split}
               & \left\lvert \frac{ S(m,n;N) }{N}   J_{k-1}\left(\frac{4\pi\sqrt{mn}}{N}\right)\right\rvert  
               \geq   2^{\omega(N)} \frac{G(N)}{N^{3/2}}  D_0^{k-1} J_{k-1}( k-1)=H(N) D_0^{k-1}J_{k-1}( k-1).
           \end{split}
       \end{equation}  
We now find a lower bound for $\left|\Delta_{k,N}(m,n)-\delta(m,n)\right|$. 
 For all $k \geq k_0(N),$ we have
 \begin{equation*}
 k-1\geq 1+ \left\lfloor\frac{\log 7-\log H(N)}{\log A_0}\right\rfloor \geq \frac{\log 7-\log H(N)}{\log A_0} ,
 \end{equation*}
 that is, $\left( D_0 e^{-H_0} \right)^{k-1}\geq \frac{7}{H(N)},$ and this condition implies that

\begin{equation}\label{inequalities10}
    \begin{split}
        7 e^{H_0(k-1)} J_{k-1}( k-1)\leq H(N) D_0^{k-1}J_{k-1}( k-1)
    \end{split}
\end{equation}holds.

Combining equations \eqref{|B|}, \eqref{S(m,n;N)times J_k lower bound} and \eqref{inequalities10}, we obtain the following inequality for all $k \geq k_0(N),$
\begin{equation*}
    \begin{split}
       \frac{|B|}{2\pi} \leq  3.005 e^{H_0(k-1)} J_{k-1}( k-1)\leq 7 e^{H_0(k-1)} J_{k-1}( k-1) \leq \left\lvert \frac{ S(m,n;N) }{N}   J_{k-1}\left(\frac{4\pi\sqrt{mn}}{N}\right)\right\rvert .
 \end{split}
\end{equation*}
Therefore, for all $k \geq k_0(N),$ we have

 \begin{equation*}
       \begin{split}
           \left|\Delta_{k,N}(m,n)-\delta(m,n)\right|
           =&\left|2\pi i^k \frac{S(m,n;N)}{N}J_{k-1}\left(\frac{4\pi\sqrt{mn}}{N} \right)+B\right|\\
            \geq &  2\pi\left(\left|\frac{S(m,n;N)}{N}J_{k-1}\left(\frac{4\pi\sqrt{mn}}{N}\right)\right|-\frac{|B|}{2\pi}\right)\\
           \geq & 7.99 \pi J_{k-1}(k-1) e^{H_0(k-1)},
           \end{split}
   \end{equation*}
   since $ \left(D_0 e^{-H_0} \right)^{k-1} \geq \frac{7}{H(N)}$ is true for all $k \geq k_0(N).$
\end{proof}
\begin{remark}
We recall that  Theorem~\ref{main theorem sec 1} is a variant of \cite[Theorem 1.7]{JS}. %In this theorem we allow $N$ to vary along with $k$.
The bounds have been made explicit so that one can use these results for weights at least $k_0(N).$ 
The size of the interval in which $\frac{4\pi\sqrt{mn}}{N}$ lies is $(1-D_0)k$ as compared to $2k^{\frac{1}{3}}$ of \cite[Theorem 1.7]{JS}, which weakens the lower bound to  $e^{(k-1)\left(1-\frac{4}{9}-\log(\frac{9}{5})\right)}$ as compared to $k^{-\frac{1}{3}}$.
 \end{remark}
\begin{proof}[Proof of Corollary ~\ref{cor 1}]
The statement follows from Lemma~\ref{Bessel2} and Theorem~\ref{main theorem sec 1}.
\end{proof}
\begin{remark}\label{rmrk}
    Applying Corollary ~\ref{cor 1} with its hypotheses,  $\epsilon=0.4470,$ $D_0=0.999,$  $k\geq\max( 22,k_0(N)) $ and $\frac{4\pi\sqrt{mn}}{N}\in \left[0.999(k-1), (k-1) \right], $   we get 
    $$\left| \Delta_{k,N}(m,n)-\delta(m,n)  \right|\geq    
    \frac{0.002397\pi}{(k-1)^{1/3}}e^{(k-1)\left(1-\frac{4}{9}-\log(\frac{9}{5})\right)}.  
    $$
    The following corollary is immediate.   Although Corollary ~\ref{D1n} is mentioned for a specific value of $D_0=0.999,$ this result holds for any $D_0 \in (e^{H_0},1),$ where $H_0$ is mentioned in Theorem~\ref{main theorem sec 1}.
\end{remark}
\begin{corollary}\label{D1n}
Let $k,\,N$ be natural numbers with $k \geq \max(22,k_0(N)),$  $k$ even and $N$ odd. Let $n\geq 2$ be a natural number such that $n\in \left[ \left(\frac{0.999( k-1)N}{4\pi}\right)^2, \left( \frac{(k-1) N}{4\pi}\right)^2 \right]$ with $(n,N)=1$, and $n$ is a quadratic residue modulo each prime dividing the powerful part of $N$. Then we have
\begin{equation*}
    \left| \Delta_{k,N}(1,n)\right|\geq   
    \frac{0.002397\pi}{(k-1)^{1/3}}e^{(k-1)\left(1-\frac{4}{9}-\log(\frac{9}{5})\right)}. 
\end{equation*}
\end{corollary}
\begin{remark}
Keeping equation \eqref{justdelta(1,n)} in mind, we note that Corollary ~\ref{D1n} gives us a lower bound for a weighted trace of Hecke operator $T_n.$
\end{remark}

Now we give a proof of the Theorem~\ref{main theorem sec 2}.
\begin{proof}[\textbf{Proof of Theorem~\ref{main theorem sec 2}}]
    The proof is similar to the proof of Theorem~\ref{main theorem sec 1}. We briefly outline the key points that need to be addressed. Since $k\geq 28,$ we see that equation \eqref{8/9} is satisfied. 
Now using \cite[Theorem 1]{KI}, we get 
\begin{equation*}
    J_{k-1}\left( (k-1)+t(k-1)^{\frac{1}{3}} \right)=\frac{2^{\frac{1}{3}}Ai(-2^{\frac{1}{3}}t)}{(k-1)^\frac{1}{3}}+\theta\frac{(4t^{\frac{9}{4}}+21)}{7(k-1)},
\end{equation*}
  where $|\theta|\leq 1$ and $0\leq t\leq 1.$ Note that for $k \geq 180,$
  $$\left|\frac{\theta(4t^{\frac{9}{4}}+21)}{7(k-1)}\right| = \left|\frac{\theta(4t^{\frac{9}{4}}+21)}{7(k-1)^{1/3}(k-1)^{2/3}}\right|\leq \frac{25}{7\times 179^{2/3} (k-1)^{1/3}} \leq \frac{0.113}{(k-1)^{1/3}}$$ and $2^{\frac{1}{3}}  Ai(-2^{\frac{1}{3}}t)\geq 2^{\frac{1}{3}} Ai(0)\geq 0.44, $ so that 
  $$ J_{k-1}\left( (k-1)+t(k-1)^{\frac{1}{3}} \right)\geq  \frac{0.327}{(k-1)^{\frac{1}{3}}} .$$
  The above inequality plays the role of equation \eqref{D0} in the proof of Theorem~\ref{main theorem sec 1}.
 Using Lemma~\ref{Bessel2}, $J_{k-1}(k-1)\leq \frac{0.448}{(k-1)^{\frac{1}{3}}}+\frac{0.095}{(k-1)^\frac{1}{3}}=\frac{0.543}{(k-1)^{\frac{1}{3}}}.$ Hence, for all $k \geq k_{1}(N),$ we have 
  \begin{equation*}
    \begin{split}
       \frac{|B|}{2\pi} \leq  3.005 e^{H_0(k-1)} J_{k-1}( k-1)\leq e^{H_0(k-1)} \frac{1.632}{(k-1)^{\frac{1}{3}}} \leq  \frac{0.327H(N)}{(k-1)^{\frac{1}{3}}} \leq \left\lvert \frac{ S(m,n;N) }{N}   J_{k-1}\left(\frac{4\pi\sqrt{mn}}{N}\right)\right\rvert .
 \end{split}
\end{equation*}
Using equation \eqref{|B|}, we obtain the following for $k \geq k_1(N).$  
\begin{equation*}
       \begin{split}
           &\left|\Delta_{k,N}(m,n)-\delta(m,n)\right|\\
           =&\left|2\pi i^k \frac{S(m,n;N)}{N}J_{k-1}\left(\frac{4\pi\sqrt{mn}}{N} \right)+B\right|\\
            \geq & 2 \pi \left|  \frac{0.327H(N)}{(k-1)^{\frac{1}{3}}} - e^{H_0(k-1)} \frac{1.632}{(k-1)^{\frac{1}{3}}}\right| \\
           \geq & \frac{2\pi E_0}{(k-1)^{\frac{1}{3}}}.
           \end{split}
   \end{equation*}
The last inequality is true if we have the following for some positive constant $E_0:$
$$ \frac{0.327H(N)}{(k-1)^{\frac{1}{3}}} -   e^{H_0(k-1)} \frac{1.632}{(k-1)^{\frac{1}{3}}}\geq\frac{E_0}{(k-1)^{\frac{1}{3}}},$$
that is,
$$
0.327H(N)-E_0\geq  e^{H_0(k-1)} 1.632,
$$
that is,
$$
e^{-H_0(k-1)} \geq \frac{1.632}{0.327H(N)-E_0},
$$
which is true for $k \geq k_{1}(N),$ where
$$k_{1}(N)= 2+\left\lfloor \frac{\log (1.632)-\log (0.327H(N)-E_0)}{-H_0}\right \rfloor.$$
\end{proof}
  \begin{remark}
      Note that the above theorem can be extended similarly to   $\left[ (k-1), (k-1)+\gamma(k-1)^{\frac{1}{3}} \right]$ with $-2^{\frac{1}{3}}\gamma>a $, where $a=-2.33811\dots$ is the largest zero  of the airy function. This forces $\gamma<1.86.$  Furthermore, it is important to note that we can not apply \cite[Theorem 1]{KI} directly for the interval $\left[  2,(k-1) \right)$ as \cite[Theorem 1]{KI} is valid for $t\geq 0.$ 
  \end{remark}  
\subsection*{Tables for \texorpdfstring{$H(N),k_0(N), k_1(N)$}{}}
In this subsection, we consider some specific values of $N$ and compute the values of  $H(N),k_0(N), k_1(N)$.
Following Remark \ref{rmrk}, we take  $D_0=0.999$ so that 
$A_0=D_0e^{-H_0}\approx 1.0317$ as given in Theorem~\ref{main theorem sec 1}. We recall that $k_0(N)$ and $k_1(N)$ appear in Theorem~\ref{main theorem sec 1} and Theorem~\ref{main theorem sec 2}  respectively.  
Recall that  $N=du$ with $d$ is powerful and $u$ is squarefree, where $(d,u)=1$. 
We have  $H(N)=2^{\omega(N)} \frac{G(N)}{N^{3/2}}$ with $G(N)= \frac{u}{(\tau(u)^{2} u)^{\phi(u)/2\tau(u)}}.$ 
If $N=p^\alpha$ for some prime $p$ and $\alpha>1$, then $H(N)=\frac{2}{p^{\frac{3\alpha}{2}}}.$ If $N$ is squarefree, then $H(N)=\frac{2^{\omega(N)}}{N^{1/2}(\tau(N)^{2}N)^{\phi(N)/2\tau(N)}}.$ If $N=p$ for a prime number, then $H(p)=\frac{2}{p^{1/2}(4p)^{(p-1)/4}}$.
  
We calculate the following table for $C_0=0.326$, $-H_0\geq  0.3223$ and  $D_{0}=0.999.$ Then, $\log A_0 \geq 0.031229.$ 

%\JMcom{There is a change in this table:}

  \begin{center}
\begin{tabular}{|c|c|c|c|}
\hline
    $N$ & $\log H(N)(\geq )$ & $k_{0}(N)= 2+ \left\lfloor\frac{\log 7-\log H(N)}{\log A_0}\right\rfloor(\leq )$  & $k_{1}(N)=2+\left\lfloor \frac{\log (1.632 /C_0)-\log H(N)}{-H_0}\right\rfloor(\leq)$  \\
\hline 
  $1$ &  0  & 64 & 6\\
  \hline 
  $3$ &  -1.0987   &  99 & 10\\
  \hline 
  $5$ &  -3.1074  & 163 & 16\\
    \hline 
  $9$ & -2.6027   & 147 & 15\\
    \hline 
  $49$ & -5.1446   &  229 & 22\\
    \hline 
  $99$ & -12.5690 & 466 & 45\\
    \hline 
  $999$ &   -71.6034 & 2357 & 229\\
    \hline 
  $9999$ &  -1227.9235 & 39384 & 3816\\
    \hline 
  $49999$ &  -152574.2718 &  4885724 & 473399      \\
    \hline 
  $99999$ &  -16325.0518  &  522817& 50658       \\
    \hline 
  $999999$ &  -15030.9423  & 481377 & 46643       \\
\hline
\end{tabular}
\end{center}

\subsection*{Acknowledgements}
 The third-named author acknowledges the support from IISER Pune, Bhaskaracharya Pratishthana (Pune), and NISER Bhubaneswar and his work was also supported at NISER Bhubaneswar by the Department of Atomic Energy (DAE project number RIN-4001).  The second-named author is grateful to COEP Technological University, Pune, for providing an excellent working environment. 
The second  and third named authors would like to thank Kaneenika Sinha for her valuable input. 
We are thankful to the Harish-Chandra Research Institute, Prayagraj, for excellent hospitality, where the parts of the work were carried out. We are grateful to the anonymous referee for pointing out an error in Theorem~\ref{Main1} and consequent results.
\subsection*{Data availability} The authors declare that the manuscript has no associated data.
\bibliographystyle{alpha}
\bibliography{Sato-TateNT.bib}
\end{document}